\documentclass[reqno,a4paper,12pt]{amsart} 

\usepackage{amsmath,amscd,amsfonts,amssymb}
\usepackage{mathrsfs,dsfont}

\numberwithin{equation}{section}
\numberwithin{figure}{section}

\newcommand\R{\mathbb{R}}

\newcommand\Z{\mathbb{Z}}

\newcommand\T{\mathbb{T}}
\newcommand\al{\alpha}

\newcommand\gam{\gamma}

\newcommand\del{\delta}
\newcommand\Del{\Delta}
\newcommand\lam{\lambda}
\newcommand\Lam{\Lambda}

\newcommand\sig{\sigma}

\newcommand\eps{\varepsilon}

\renewcommand\S{\mathcal{S}}

\renewcommand\le{\leqslant}
\renewcommand\ge{\geqslant}
\renewcommand\leq{\leqslant}
\renewcommand\geq{\geqslant}
\newcommand\sbt{\subset}

\newcommand{\ft}[1]{\widehat{#1}}

\newcommand{\supp}{\operatorname{supp}}
\newcommand{\spec}{\operatorname{spec}}

\newcommand{\sign}{\operatorname{sign}}

\newcommand{\norm}[2]{\|{#1}\|_{{#2}}}
\newcommand{\nmb}[1]{\norm{{#1}}{*}}

\theoremstyle{plain}
\newtheorem{thm}{Theorem}[section]
\newtheorem{lem}[thm]{Lemma}
\newtheorem{lemma}[thm]{Lemma}

\newtheorem*{claim*}{Claim}

\newcommand{\thmref}[1]{Theorem~\ref{#1}}
\newcommand{\secref}[1]{Section~\ref{#1}}

\newcommand{\lemref}[1]{Lemma~\ref{#1}}

\theoremstyle{definition}

\newtheorem*{definition*}{Definition}
\newtheorem*{remarks*}{Remarks}
\newtheorem*{remark*}{Remark}
\newtheorem{remark}[thm]{Remark}

\newenvironment{enumerate-alph}
{\begin{enumerate}
\addtolength{\itemsep}{5pt}
}
{\end{enumerate}}

\newenvironment{enumerate-num}
{\begin{enumerate}
\addtolength{\itemsep}{5pt}
}
{\end{enumerate}}

\newenvironment{enumerate-text}
{\begin{enumerate}
\addtolength{\itemsep}{5pt}
}
{\end{enumerate}}

\begin{document}

\title{Schauder frames of discrete translates in $L^p(\mathbb{R})$}

\author{Nir Lev}
\address{Department of Mathematics, Bar-Ilan University, Ramat-Gan 5290002, Israel}
\email{levnir@math.biu.ac.il}

\author{Anton Tselishchev}
\address{Department of Mathematics, Bar-Ilan University, Ramat-Gan 5290002, Israel}
\address{St. Petersburg Department of Steklov Mathematical Institute, Fontanka 27, St. Petersburg 191023, Russia}
\email{celis\_anton@pdmi.ras.ru}

\date{December 10, 2025}
\thanks{Research supported by ISF Grants No.\ 1044/21 and 854/25}

\begin{abstract}
For every $p > (1 + \sqrt{5})/2$ we construct a uniformly discrete real sequence $\{\lambda_n\}_{n=1}^\infty$ satisfying $|\lambda_n| = n + o(1)$, a function $g \in L^p(\mathbb{R})$, and continuous linear functionals $\{g^*_n\}_{n=1}^\infty$ on $L^p(\mathbb{R})$, such that every $f \in L^p(\mathbb{R})$ admits a series expansion
\[
f(x) = \sum_{n=1}^{\infty} g_n^*(f) g(x-\lambda_n)
\]
convergent in the $L^p(\mathbb{R})$ norm. We moreover show that $g$ can be chosen nonnegative.
\end{abstract}

\maketitle


\section{Introduction}
\label{secA1}

\subsection{}
A system of vectors $\{x_n\}_{n=1}^{\infty}$ in a Banach space $X$
is called a \emph{Schauder basis} if every  $x \in X$  admits a unique series expansion
$x = \sum_{n=1}^{\infty} c_n x_n$ where $\{c_n\}$ are scalars.
 It is well known that in this case there exist 
 biorthogonal continuous linear functionals $\{x_n^*\}$
 such that  the coefficients of the series expansion are given by
 $c_n = x_n^*(x)$   (see e.g.\  \cite[Section 1.6]{You01}).
 If the series converges unconditionally
 (i.e.\ if it converges for any rearrangement of its terms)
for every  $x \in X$, then $\{x_n\}$ 
is said to be an \emph{unconditional} Schauder basis.

Given a function $g \in L^p(\R)$, we denote its translates by
\begin{equation}
\label{eqTL}
(T_\lambda g)(x) = g(x-\lambda), \quad \lam \in \R.
\end{equation}
There is a long-standing open problem, asking
whether the space $L^p(\R)$, $1<p<\infty$, admits
a Schauder basis formed by translates of a single function
(see \cite{OZ92}, \cite[Problem 4.4]{OSSZ11}).
It is known that \emph{unconditional} Schauder bases consisting of translates 
do not exist in any of these spaces, see \cite{OZ92},
 \cite{OSSZ11}, \cite{FOSZ14}.

A sequence $\Lam = \{\lam_n\}_{n=1}^{\infty}$ of real numbers
is said to be \emph{uniformly discrete} if 
\begin{equation}
\label{eq:ud}
\inf_{n \neq m}  |\lam_m - \lam_n| > 0.
\end{equation}
It was observed in \cite[Theorem 1]{OZ92}
that the condition \eqref{eq:ud} is necessary for 
 a system  of translates 
$\{T_{\lam_n} g \}_{n=1}^{\infty}$
to form a Schauder basis in $L^p(\R)$.
It is also known
 that in the space $L^1(\R)$, 
a system of uniformly discrete 
translates cannot even be complete,
see  \cite{BOU06}.
Hence, no Schauder bases of translates 
exist  in $L^1(\R)$.

\subsection{}
If $X$ is a Banach space with dual space $X^*$, then a system
$\{(x_n, x_n^*)\}_{n=1}^\infty$ in $X \times X^*$ is called a 
\emph{Schauder frame} (or a quasi-basis) if every $x\in X$ has a series expansion
\begin{equation}
\label{def_Sch_frame}
x=\sum_{n=1}^\infty x_n^* (x) x_n.
\end{equation}
If the series \eqref{def_Sch_frame}
converges unconditionally for every  $x \in X$, then 
$\{(x_n, x_n^*)\}$ is called an \emph{unconditional} Schauder frame.
We note that  the series expansion 
\eqref{def_Sch_frame} need not be unique  and the coefficient functionals $\{x_n^*\}$ 
 need not be biorthogonal to $\{x_n\}$.	Hence 
Schauder frames form a wider class of representation systems than Schauder bases.

It was shown in \cite{FOSZ14} that for every $p>2$ 
there exists an unconditional Schauder frame
in the space $L^p(\R)$ consisting of translates, i.e.\ of the form
$\{(T_{\lam_n} g, g_n^*)\}$ where $g \in L^p(\R)$,
$\{\lam_n\} \sbt \R$  and
$\{g_n^*\}$ are continuous linear functionals on $L^p(\R)$.
Moreover, $\{\lam_n\}$ may be chosen to be
 \emph{an arbitrary unbounded sequence}, and
 in particular, it may consist of integers, and  
 may increase arbitrarily fast.

To the contrary, we proved recently \cite{LT25a} that 
if $1 \le p \le 2$, then the space  $L^p(\R)$
does not admit an unconditional Schauder frame
consisting of translates.

In \cite[Section 4]{FPT21} a construction was given
of Schauder frames (not unconditional) 
 of translates  in $L^p(\R)$, $1 \le p < \infty$.
In fact, it was proved 
that whenever a system of translates $\{T_\lam g\}_{\lam \in \Lam}$
is complete in the space $L^p(\R)$,  then there exists a Schauder frame
$\{(T_{\lam_n} g, g_n^*)\}_{n=1}^{\infty}$ 
 such that $\{ \lam_n \} \sbt \Lam$.
 However the Schauder frames obtained by this construction 
are \emph{highly redundant}, as the sequence
 $\{ \lam_n \}$ is composed
of countably many blocks of finite size,
such that each block gets repeated a higher and higher
 number of times. The  sequence
 $\{ \lam_n \}$ thus ``runs back and forth'' through the set $\Lam$.

\subsection{}
The following question was posed in \cite[Problem 4.4]{OSSZ11}:
does there exist  a Schauder frame in $L^p(\R)$ formed by 
a \emph{uniformly discrete} sequence of translates?
The problem has remained open for  $1 < p \le 2$.
We recently obtained \cite{LT23} an affirmative answer for  $p=2$,
i.e.\ we constructed in the space $L^2(\R)$
a Schauder frame of the form
$\{(T_{\lam_n} g, g_n^*)\}_{n=1}^{\infty}$ 
where $\{\lam_n\}_{n=1}^{\infty}$ is a uniformly discrete real sequence.

The proof in \cite{LT23}
was based on the fact that the Fourier transform is a unitary mapping 
$L^2(\R) \to L^2(\R)$. This point breaks down
for $L^p(\R)$ spaces, $1 <p < 2$.

The main goal of the present paper is to extend the  result from  \cite{LT23}
to certain values
of $p$ within the range $1<p<2$. We will prove the following result:

\begin{thm}
\label{thm:main}
Let $p > (1+\sqrt{5})/2$. Then there exist a uniformly discrete
real sequence $\{\lambda_n\}_{n=1}^\infty$ satisfying
$|\lam_n| = n + o(1)$, 
a function $g \in L^p(\R)$ and a sequence $\{g^*_n\}_{n=1}^\infty$
in $ (L^p(\R))^*$,
such that every $f\in L^p(\R)$ admits a series expansion
\begin{equation}
\label{eq:F1.1}
f(x) = \sum_{n=1}^\infty g^*_n(f) g(x - \lambda_n)
\end{equation}
convergent in the $L^p(\R)$  norm.
\end{thm}

The question whether the result holds for every $p>1$, remains open.

\subsubsection*{Remarks}
1. The result can be strengthened in various directions;
see \cite[Section 4]{LT23} where several extensions of the result 
for $p=2$ are given, some of which may still be valid
for the result of the present paper.

2. Motivated by recent interest in nonnegative
coordinate systems in $L^p(\R)$ spaces, see
\cite{PS16},  \cite{JS15}, \cite{FPT21},
one may ask whether there exists a Schauder 
frame  in the space $L^p(\R)$  formed 
by a uniformly discrete sequence of translates 
of some \emph{nonnegative} function $g$.
We will show 
that  by a certain modification of our
proof, this additional requirement of 
nonnegativity of $g$ can 
be achieved in \thmref{thm:main}
(see \secref{sec:nonneg}).


\section{Preliminaries}
\label{secP1}

In this section we present some necessary background
and fix notation that will be used throughout the paper.

\subsection{}
The \emph{Schwartz space}  $\S(\R)$ consists of all infinitely smooth
functions $\varphi$ on $\R$ such that for each $n,k \geq 0$, the seminorm
\begin{equation}
\|\varphi\|_{n,k} := \sup_{x \in \R} (1+|x|)^n  |\varphi^{(k)}(x)|
\end{equation}
is finite. A \emph{tempered distribution} is a 
linear functional  on the Schwartz space  
which is continuous with respect to the 
topology generated by this family of seminorms.
We use  $\alpha(\varphi)$ to denote  the action of
a  tempered distribution $\alpha$ on a
Schwartz function $\varphi$.

We denote by  $\supp(\al)$ the closed
support of a tempered distribution $\al$.

If $\varphi$ is a Schwartz function on $\R$ then we define its Fourier transform by
\begin{equation}
\ft{\varphi}(x)=\int_{\R} \varphi(t) e^{-2 \pi i xt} dt.
\end{equation}
The Fourier transform of a  tempered distribution 
$\alpha$ is defined by 
$\ft{\alpha}(\varphi) = \alpha(\ft{\varphi})$.

If $\alpha$ is a tempered distribution  and if
$\varphi$  is a 
Schwartz function, then 
the product $\alpha \cdot \varphi$ is 
a tempered distribution   defined
by $(\alpha \cdot \varphi)(\psi) =
\alpha(\varphi \cdot \psi)$,
$\psi \in \S(\R)$.
The convolution
$\alpha \ast \varphi$ 
of a tempered distribution  
 $\alpha$ and 
a Schwartz function $\varphi$ is 
 an infinitely smooth function
which is also
a tempered distribution, and
whose Fourier transform is $\ft{\al} \cdot \ft{\varphi}$.

\subsection{}
Let $A^p(\T)$, $1\le p < \infty$, denote the Banach space of 
 Schwartz distributions $\alpha$ on the circle $\T = \R / \Z$
whose Fourier coefficients $\{\ft{\alpha}(n)\}$, $n \in \Z$,
belong to $\ell^p(\Z)$, endowed with the norm 
$\|\alpha\|_{A^p(\T)} := \|\ft{\alpha}\|_{\ell^p(\Z)}$.
For $p=1$ this is the classical Wiener algebra $A(\T)$ 
of continuous functions with an absolutely convergent
Fourier series.

 We also use $A^p(\R)$, $1\le p < \infty$, to denote
 the Banach space of 
tempered distributions $\alpha$ on $\R$
whose Fourier transform $\ft{\alpha}$
is in $L^p(\R)$,   with the norm 
$\|\alpha\|_{A^p(\R)} := \|\ft{\alpha}\|_{L^p(\R)}$. 

Note that $A^1(\T)$ and $A^1(\R)$ are function spaces,
continuously embedded in $C(\T)$ and $C_0(\R)$ respectively.
Similarly, for $1 < p \le 2$ the space $A^p$ (on either $\T$ or $\R$)
is a function space, continuously embedded in $L^{q}$,
 $q = p/(p-1)$, by the Hausdorff-Young inequality.
On the other hand, $A^p$ is not a function space for $p>2$.

\subsection{}
For $0 < h < 1/2$  we denote by $\Delta_h$ the ``triangle function'' on $\T$ vanishing
outside $(-h, h)$, linear on $[-h, 0]$ and on $[0,h]$, and satisfying
$\Delta_h(0)=1$. Then $\ft{\Delta}_h(0) = h$, and
\begin{equation}
\label{estim_delta}
\|\Delta_h\|_{A^p(\T)} \leq h^{(p-1)/p}, \quad 1 \le p < \infty.
\end{equation}

Indeed, to obtain the estimate \eqref{estim_delta} one can use the fact 
that Fourier coefficients of $\Delta_h$ are real and nonnegative,
hence $\|\Delta_h\|_{A(\T)} = \sum_n  \ft{\Delta}_h(n)=\Delta_h(0)=1$.
 Moreover, we have $\ft{\Delta}_h(n) \le \int_{\T} \Delta_h(t) dt = h$ for every $n\in\Z$,
and so $\|\Delta_h\|_{A^p(\T)}^p = \sum_n  \ft{\Delta}_h(n)^p\le h^{p-1}$.

For $0 < h < 1/4$ we also use $\tau_h$ to
denote the ``trapezoid function'' on $\T$ which vanishes
outside $(-2h, 2h)$, is equal to $1$ on $[-h, h]$, and
is linear on $[-2h, -h]$ and on $[h,2h]$. Then 
  $\ft{\tau}_h(0) = 3h$, and 
\begin{equation}
\label{eq:tauhap}
\|\tau_h\|_{A^p(\T)} \leq 3 h^{(p-1)/p}, \quad 1 \le p < \infty,
\end{equation}
which follows from \eqref{estim_delta} and the fact that
$\tau_h(t) = \Delta_h(t + h) + \Delta_h(t) + \Delta_h(t-h)$.

\subsection{}
By a \emph{trigonometric polynomial} we mean a finite sum of the form
\begin{equation}
\label{eq:P1.1}
P(t) = \sum_{j} a_j e^{2\pi i \sigma_j t}, \quad t \in \R,
\end{equation}
where $\{\sigma_j\}$ are distinct real numbers, and $\{a_j\}$ are complex numbers. 

By the \emph{spectrum} of $P$ we mean the set 
$\spec(P) := \{\sigma_j : a_j \neq 0\}$.
We observe that if  $P$
  has integer spectrum, $\spec(P) \sbt \Z$,
  then $P$ is $1$-periodic, that is, 
  $P(t+1) = P(t)$. In this case, $P$ may be
  considered also as a function on $\T$.

By the \emph{degree} of $P$ we mean the number
$\deg(P) := \min \{r \ge 0 : \spec(P) \sbt [-r,r]\}$.

The (symmetric) \emph{partial sum}
$S_{r}(P)$ of a trigonometric polynomial
\eqref{eq:P1.1} is defined by 
\begin{equation}
\label{eq:P1.2}
 S_{r}(P)(t) = \sum_{   | \sig_j| \le r} a_j e^{2 \pi i \sig_j t}.
\end{equation}

We say that $P$ is \emph{analytic}
if we have $\spec(P) \sbt [0, +\infty)$.

\subsection{} 
For a trigonometric polynomial
\eqref{eq:P1.1} we use the notation
$\|\ft{P}\|_p = (\sum_j |a_j|^p)^{1/p}$
and $\| \ft{P} \|_\infty = \max_j |a_j|$.
If $f \in A^p(\R)$ and $P$ is a trigonometric polynomial, then 
\begin{equation}
\label{est_prod_2}
\|f \cdot P \|_{A^p(\R)} \le \|\ft{P} \|_{1}  \cdot \|f  \|_{A^p(\R)},
\end{equation}
which follows by an application of the triangle inequality in the space $L^p(\R)$,
using the fact that the Fourier transform of $f \cdot P$ is  given by
$\sum_{j} a_j \ft{f}(x-\sigma_j)$.

In a similar way one can establish the   inequality 
\begin{equation}
	\label{est_prod_1}
	\|f \cdot g \|_{A^p(\T)} \le \|f  \|_{A(\T)} \cdot
	\|g  \|_{A^p(\T)}, \quad f \in A(\T), \; g \in A^p(\T).
\end{equation}

\subsection{}
We introduce an auxiliary norm $\nmb{\cdot}$ on the Schwartz space $\S(\R)$, defined by
\begin{equation}
\nmb{u}  := 10 \cdot \sup_{x \in \R}  (1 + x^2) |\ft{u}(x)|, \quad u \in \S(\R).
\end{equation}

If $u \in \S(\R)$ and
$f\in  A^p(\T)$, $1 \le p < \infty$,
then $f$ may be
considered also as a $1$-periodic tempered distribution  on $\R$,
so the product $u \cdot f$ makes sense and is well defined.

\begin{lemma}
\label{ufprodst}
Let	$u \in \S(\R)$ and
$f\in  A^p(\T)$, $1 \le p < \infty$. Then
\begin{equation}
\label{est_prod_3}
\| u \cdot f \|_{A^p(\R)} \le \nmb{u}   \| f  \|_{A^p(\T)}.
\end{equation}	
\end{lemma}

\allowdisplaybreaks

Indeed,  an application of H\"{o}lder's inequality yields
\begin{align}
\| u \cdot f \|_{A_p(\R)}^p &= \int \Big| \sum_{n\in\Z} \ft{f}(n) \ft{u}(x-n) \Big|^p\, dx \\
&\leq \int \Big( \sum_{n\in\Z} |\ft{f}(n)|^p |\ft{u}(x-n)| \Big)
 \Big( \sum_{n\in\Z} |\ft{u}(x-n)| \Big)^{p-1} dx \\ 
&\leq \nmb{u}^{p-1} \int \Big( \sum_{n\in\Z} |\ft{f}(n)|^p |\ft{u}(x-n)| \Big)  dx \\ 
&\le \nmb{u}^{p-1} \nmb{u} \sum_{n\in\Z} |\ft{f}(n)|^p = \nmb{u}^p \|f\|_{A^p(\T)}^p.
\end{align}

If $f\in  A^p(\T)$ and $\nu$ is a positive integer, then we use
  $f_\nu$ to denote the element of the space
$A^p(\T)$ whose Fourier series is given by
$\sum_{n \in \Z} \ft{f}(n) e^{2 \pi i n \nu t}$.

\begin{lemma}
\label{sep-specR}
Let	$u \in \S(\R)$ and  $f\in  A^p(\T)$, $1 \le p < \infty$. Then
\begin{equation}
\label{eq_sep-specR}
\lim_{\nu\to \infty}\|u \cdot f_\nu \|_{A^p(\R)} =\|u\|_{A^p(\R)}  \|f\|_{A^p(\T)}.
\end{equation}
\end{lemma}

This   is obvious if $u$ has a compactly supported Fourier transform $\ft{u}$.
In the general case, \eqref{eq_sep-specR} 
can be proved by approximating   $u$ in the $\nmb{\cdot}$ norm
by a Schwartz function $v$ with a compactly supported Fourier transform $\ft{v}$, 
and   using the inequality \eqref{est_prod_3} to estimate the error.

\subsection{}
If  $\alpha$ is a  tempered distribution on $\R$
and $P$ is a trigonometric polynomial, then the product
$\al \cdot P$ is  a  tempered distribution 
defined by $(\al \cdot P)(\varphi)
= \al (P \cdot \varphi)$, $\varphi \in \S(\R)$.

\begin{lemma}
\label{lem:uniqaqprod}
Let $\alpha \in  A^p(\R)$, $1 \le p < \infty$, and 
suppose that $P$ is a nonzero trigonometric polynomial on $\R$.
If $\alpha \cdot P = 0$, then $\alpha = 0$.
\end{lemma}

\begin{proof}
The condition $\alpha \cdot P  = 0$
implies that $\supp(\alpha)$ is contained in the
set of zeros of $P$, which is a discrete closed set in $\R$.
Let $\chi$ be a Schwartz function on $\R$ with
$\chi(t) =1$ in a neighborhood of a point
$a \in \supp(\alpha)$, and 
$\chi(t) =0$ in a neighborhood of $\supp(\alpha) \setminus \{a\}$.
 The distribution ${\al} \cdot {\chi}$ is then
 supported at the point $a$ and
 coincides with ${\al}$ in a neighborhood of $a$.
  It is well known that a distribution supported at a single
  point $a$  is a finite linear combination of derivatives
  of Dirac's measure at the point $a$. In turn, this
  implies that $(\ft{\al} \ast \ft{\chi})(x) = e^{-2 \pi i a x} q(x)$
  where $q$ is a polynomial.
    But $\ft{\al} \ast \ft{\chi} \in L^p(\R)$, so this is possible only if 
      $\ft{\al} \ast \ft{\chi}$ is zero, and hence
      $\al \cdot \chi$ is zero. We conclude that ${\al}$ must
vanish in a neighborhood of any point $a \in \supp(\al)$,
a contradiction unless  $\alpha = 0$.
\end{proof}

\subsection{} \label{subsec:complete}
We will use the known fact that 
in the space $L^p(\R)$, $p>1$, there exist
\emph{complete systems} formed by uniformly discrete
translates of a single function.

It was proved in \cite{AO96} that
for every $p>2$, there is a function $g \in L^p(\R)$
whose translates  by the positive integers 
$\{g(x-n)\}$, $n=1,2,3,\dots$,
span the whole space $L^p(\R)$, that is,
these translates are complete in  $L^p(\R)$.
A similar result was proved also in the space $C_0(\R)$.
On the other hand, no system of integer translates 
can be complete in $L^p(\R)$ for $1 \le p \le 2$
 (see e.g.\  Example 11.2 and Corollary 12.26 in \cite{OU16}).

However, it was proved in
\cite{Ole97} that for any ``small perturbation'' 
of the integers,
\begin{equation}
\label{eq:pertz}
\lam_n = n + \alpha_n, \quad  
0 \ne \alpha_n \to 0 \quad (|n| \to +\infty)
\end{equation}
there exists  $g \in L^2(\R)$
 such that  the system
$\{g(x-\lam_n)\}$, $n \in \Z$,
is complete in $L^2(\R)$. It was moreover
 shown in \cite{OU04} 
that if the perturbations are exponentially small, i.e.\
\begin{equation}
\label{eq:exppertn}
0 < |\alpha_n | < C r^{|n|}, \quad n \in \Z,
\end{equation}
for some $0<r<1$ and $C>0$, 
then  $g$ can be chosen in the Schwartz class.

More recently, by a development of the approach
from \cite{OU04}, the latter result was
extended  to $L^p(\R)$ spaces \cite{OU18a},  namely,
it was proved that there is a Schwartz  function
 $g$ such that if the sequence
 $\{ \lam_n \}$, $n \in \Z$, 
satisfies \eqref{eq:pertz} and \eqref{eq:exppertn} 
then the system  $\{g(x-\lam_n)\}$, $n \in \Z$,
is complete in $L^p(\R)$ for every $p>1$
(see also \cite{OU18b}).

In a recent paper \cite{Lev25}, a different approach was given
for constructing a function $g$ which spans the
space  $L^p(\R)$, $p>1$, by uniformly discrete translates.
In fact, this approach allows to use only positive translates,
and  moreover the completeness remains true for any subsystem 
obtained by the removal of a finite number of elements.

\begin{thm}[{see \cite[Theorem 1.1]{Lev25}}]
\label{thm:lev24}
There is a real sequence 
$\{\lam_n\}_{n=1}^{\infty}$ satisfying
$ \lam_n = n +  o(1)$,  and there is 
a Schwartz function $g$ on $\R$, 
such that for any $N$ the system 
\begin{equation}
\label{eq:ftranslama1}
\{g(x - \lam_n)\}, \; n>N,
 \end{equation}
is complete in the space $L^p(\R)$ for every $p>1$.
\end{thm}


\section{Localization lemma}
\label{secP2}

\subsection{}
A key ingredient in our proof of \thmref{thm:main} is the following lemma.

\begin{lem}
\label{lemP3.3}
Let $p > (1+\sqrt{5})/2$. 
Given any $\eps > 0$ there exist two real trigonometric polynomials
$P$  and  $\gam$ with integer spectrum, such that 
\begin{enumerate-num}
\item \label{tlo:i} $\ft{P}(0) = 0$,  $\| \ft{P} \|_{\infty} < \eps$;
\item \label{tlo:ii} $\| \gam  - 1 \|_{A^p(\T)} < \eps$;
\item \label{tlo:iii} $\|\gam \cdot P - 1 \|_{A^p(\T)} < \eps$;
\item \label{tlo:iv} $\max\limits_l \|\gam \cdot S_l(P)  \|_{A^p(\T)} < C_p$,
\end{enumerate-num}
where $C_p$ is a constant depending only on $p$.
\end{lem}

\begin{remarks*}
1. We note that the conditions 
\ref{tlo:i} and
\ref{tlo:ii}+\ref{tlo:iii} ``go against'' each other, as
\ref{tlo:i} says that $P$ is ``small'', while
\ref{tlo:ii}+\ref{tlo:iii} imply that $P$ should be 
``nearly'' one. For instance, it is easy to see
that the lemma fails for $p=1$,
since in this case \ref{tlo:ii}+\ref{tlo:iii} imply
that $P$ must be uniformly close to one, 
and hence  $\ft{P}(0)$ cannot be small.

2. Moreover, according to a recent preprint \cite{Boc25},
the lemma fails for $p < (1+\sqrt{5})/2$.
While this leaves open the question of whether  \thmref{thm:main} holds 
for every $p>1$,  it indicates that
proving this would require new ideas.

3. The case $p = 2$ is simpler since we have $A^2(\T) = L^2(\T)$
by Parseval's theorem; in this case the lemma follows
from \cite[Lemma 2.2]{Ole02}. However, for
$p < 2$ the Parseval theorem is no longer available, and
our proof below requires some additional 
ideas and a more careful analysis of the $A^p(\T)$ norm.

4. The proof of  \lemref{lemP3.3} given below 
establishes condition \ref{tlo:iv} with
an absolute constant $C_p$ which in fact  does not depend on $p$.
\end{remarks*}

\subsection{}
The following assertion will be used in our proof of \lemref{lemP3.3}.

\begin{lemma}
\label{polydil}
Let	$P_0, \dots, P_{N-1}$ be trigonometric polynomials 
with integer spectrum, and let $\nu$ be a positive
integer, $\nu > 2 \deg (P_j)$, $0 \le j \le N-1$.
Define
$P(t) := \prod_{j=0}^{N-1} P_j(\nu^j t)$.
Then we have
$\ft{P}(0) = \prod_{j=0}^{N-1} \ft{P}_j(0)$
and
$\|P\|_{A^p(\T)} =  \prod_{j=0}^{N-1} \|P_j\|_{A^p(\T)}$.
\end{lemma}

Indeed, expanding each $P_j$ as a Fourier sum yields
\begin{equation}
\label{eqPexpand}
P(t) =  \sum \Big[ \prod_{j=0}^{N-1} \ft{P}_j(k_j) \Big]
e^{2 \pi i t \sum_{j=0}^{N-1} k_j \nu^j}
\end{equation}	
where the sum goes through all integer vectors
$(k_0, k_1, \dots, k_{N-1})$
with $|k_j| \le \deg(P_j)$.
The condition $\nu > 2 \deg (P_j)$, $0 \le j \le N-1$,
ensures that the exponentials in 
\eqref{eqPexpand} have distinct frequencies, so that
\eqref{eqPexpand} is the Fourier expansion of $P$.
The conclusion of the lemma now follows 
in a straightforward manner.

\subsection{Proof of \lemref{lemP3.3}}

The idea of the proof is inspired by the
``separation of spectra'' technique, see \cite[Section 2.3]{KO01}.
We decompose the polynomial $P$ as a sum of small
elementary pieces, whose Fourier spectra are
localized on disjoint intervals.
In turn, the polynomial $\gam$ is obtained as a
Riesz-type product, which ensures that the set
where $P$ attains large values is essentially
localized away from the support of $\gam$.

We now turn to the details of the proof.
It is divided into several steps.

\subsubsection{}
Due to monotonicity of the $A^p$  norms, 
we may assume  that $ p < 2$.
Let us choose $\del = \del(\eps, p) > 0$ small enough, to be specified later.
We then choose and fix $0 < h < 1/3$ and a positive integer $N$ 
(both depending on $\eps$, $\del$ and $p$) such that 
\begin{equation}
\label{eq:mainconlem}
N > \eps^{-1}, \quad (1 + 3^p h^{p-1})^N < 1 + \del, \quad
(1-3h)^N > 1 - \del, \quad 4^p N^{-p} h^{-1}  < 1 - \del.
\end{equation}	

We show that such a choice of $h$ and $N$ exists. Indeed, 
denote $\eta := 3^{-p} \log(1+\del)$, and let $h = h(N,\eta,p)$
be defined by the condition $N = \eta \cdot h^{1-p}$.
Then for $N$ sufficiently large we have the inequalities 
$N>\eps^{-1}$;
$(1 + 3^p h^{p-1})^N <   \exp(3^p \eta) = 1+\del$;
$(1-3h)^N > 1 - 3hN  > 1 - \del$, and lastly, using
the assumption 
 $p  > (1 + \sqrt{5})/2$, we can also ensure that
$N^{-p} h^{-1} = \eta^{-p} h^{p(p-1)-1} < 4^{-p}(1-\del)$.

\subsubsection{}
Now  observe that 
$(1 - \tau_h) \cdot (1  - h^{-1} \Del_h) = 1 - \tau_h$. 
Let  $f$ and $g$ be  Fourier partial sums 
of $ 1 -  \tau_h$ and $ 1 - h^{-1} \Delta_h$ respectively,  
of sufficiently high order such that
\begin{equation}
\label{eq:fgproda}
\|f \cdot g - f \|_{A} < \del.
\end{equation}		
Next,  choose a  positive integer $\nu$ satisfying 
\begin{equation}
\label{eq:nularge}
\nu > 2( \deg(f)+ \deg(g)),
\end{equation}		
and define
\begin{equation}
\gam(t) := \prod_{j=0}^{N-1} f(\nu^j t), \quad
P(t) := \frac1{N} \sum_{j=0}^{N-1} g(\nu^j t).
\end{equation}		
We will check that the
conditions	\ref{tlo:i}--\ref{tlo:iv} are satisfied.

\subsubsection{}
First we note that $\ft{P}(0) = 0$. Also,
 due to \eqref{eq:nularge},
\begin{equation}
\| \ft{P} \|_{\infty} = \frac1{N} \| \ft{g} \|_{\infty} \le
\frac1{N} \sup_{n \neq 0} h^{-1} \ft{\Delta}_h(n) \le 
\frac1{N} \int_{\T} h^{-1} \Delta_h(t) dt = \frac1{N} < \eps,
\end{equation}
where the last inequality is due to \eqref{eq:mainconlem}.
Thus we obtain condition \ref{tlo:i}.

\subsubsection{}
Next, due to \eqref{eq:tauhap} we have
\begin{equation}
\label{eq:lemfapest}
 \| f \|_{A^p}^p \le \|1 - \tau_h\|_{A^p}^p = 
(1 - 3h)^p + \sum_{n \ne 0} |\ft{\tau}_h(n)|^p 
< 1 + 3^p h^{p-1},
\end{equation}
hence using \eqref{eq:mainconlem}, \eqref{eq:nularge}   we obtain
\begin{equation}
\label{eq:gamapp}
 \| \gam\|_{A^p}^p = \big( \| f \|_{A^p}^{p} \big)^N  < (1 + 3^p h^{p-1})^N < 1 + \del.
\end{equation}		
 Also, again using \eqref{eq:mainconlem}, \eqref{eq:nularge} we have
\begin{equation}
\label{eq:gamzeroft}
\ft{\gam}(0) = \ft{f}(0)^N =  (1-3h)^N > 1 - \del,
\end{equation}
and so it follows from \eqref{eq:gamapp}, \eqref{eq:gamzeroft} that
\begin{equation}
\label{eq:aponegam}
 \| \gam - 1\|_{A^p}^p =  \| \gam\|_{A^p}^p - \ft{\gam}(0)^p + (1 - \ft{\gam}(0))^p
< (1 + \del) - (1 - \del)^p + \del^p < (\eps/2)^p,
\end{equation}
provided that $\del = \del(\eps, p)$ is sufficiently small.
So condition \ref{tlo:ii} follows.

\subsubsection{}
Next, we have
\begin{equation}
\label{eq:gamptdiff}
  \gam(t)  (P(t) - 1)= \frac1{N} \sum_{j=0}^{N-1} 
(f(\nu^j t) g(\nu^j t) - f(\nu^j t)) \prod_{k \neq j} f(\nu^k t),
\end{equation}
and thus, recalling \eqref{eq:fgproda}, \eqref{eq:nularge},
\eqref{eq:lemfapest}, \eqref{eq:gamapp}, this implies  
\begin{equation}
\label{eq:gampsimgam}
\|\gam \cdot (P - 1)  \|_{A^p}
\le \|f \cdot g - f \|_{A} \cdot \big( \| f \|_{A^p} \big)^{N-1} 
 < \del (1 + \del)^{1/p} < \eps/2,
\end{equation}
for $\del = \del(\eps, p)$ small enough.
We conclude from  \eqref{eq:aponegam}, \eqref{eq:gampsimgam} that
\begin{equation}
\| \gam \cdot P  - 1\|_{A^p}  \le 
\| \gam   \cdot (P  - 1) \|_{A^p} +
\| \gam - 1 \|_{A^p} < \eps/2 + \eps/2 = \eps,
\end{equation}
and thus condition \ref{tlo:iii} holds.

\subsubsection{}
Let us finally check that also
condition \ref{tlo:iv} is satisfied.
Any partial sum $S_l(P)$ can be decomposed as 
$S_l(P)(t) = A(t)  + B(t)$, where 
\begin{equation}
A(t) := \frac1{N} \sum_{j=0}^{s-1} g(\nu^j t), \quad
B(t) := \frac1{N} S_m(g)(\nu^s t).
\end{equation}
Using the same argument as in 
\eqref{eq:gamptdiff}, \eqref{eq:gampsimgam} we can obtain
\begin{equation}
\label{eq:asnprodgam}
\| \gam \cdot (A  - s/N ) \|_{A^p}
\le \frac{s}{N} \cdot \|  f \cdot g - f \|_{A}   \cdot \big( \| f \|_{A^p} \big)^{N-1} 
< \del (1 + \del)^{1/p} < \eps/2,
\end{equation}
and consequently 
\begin{equation}
\label{eq:asimgam}
 \| \gam \cdot A \|_{A^p} \le 
\| \gam \cdot (A  - s/N ) \|_{A^p}
+ \frac{s}{N}  \cdot \|  \gam \|_{A^p} 
< \eps/2  +  (1 + \eps)  < 2.
\end{equation}
Next, we have
\begin{equation}
\label{eq:gbexpand}
 \gam(t) B(t) = 
 \frac1{N}   (f \cdot  S_m(g))(\nu^s t) \prod_{k \neq s} f(\nu^k t),
\end{equation}
and therefore due to \eqref{eq:nularge} 
\begin{equation}
\label{eq:bprgam}
 \| \gam \cdot B \|_{A^p}
=  \frac1{N}  \| f \cdot S_m(g) \|_{A^p} \cdot \big( \| f \|_{A^p} \big)^{N-1}.
\end{equation}
We note using \eqref{eq:tauhap} that
\begin{equation}
\label{eq:faleqfour}
 \| f \|_{A} \le \| 1 -  \tau_h \|_A  \le 1 + \| \tau_h \|_A \le  4,
\end{equation}
and due to \eqref{estim_delta},
\begin{equation}
\label{eq:smgapleq}
 \| S_m(g) \|_{A^p}^p \le  \| g \|_{A^p}^p 
 \le \| 1 - h^{-1} \Delta_h \|_{A^p}^p
 = h^{-p} \sum_{n \neq 0}  \ft{\Delta}_h(n)^p < h^{-p} h^{p-1} = h^{-1}.
\end{equation}
Therefore, using
\eqref{eq:bprgam},
\eqref{eq:faleqfour},
\eqref{eq:smgapleq}
and \eqref{eq:mainconlem},
\begin{equation}
\label{eq:bprgmest}
 \| \gam \cdot B \|_{A^p}^p \le 4^p N^{-p}  h^{-1} (1 + \del)
 < (1 - \del)(1 + \del)  < 1.
\end{equation}
We conclude that
\begin{equation}
\| \gam \cdot S_l(P) \|_{A^p} \le
\| \gam \cdot A \|_{A^p} +
 \| \gam \cdot B \|_{A^p} < 2 + 1 = 3.
\end{equation}
Thus condition \ref{tlo:iv} is established
and the lemma is proved.
\qed


\section{Schauder frames of weighted exponentials}
\label{secC1}

In this section we prove \thmref{thm:main}.
First we note that the
Fourier transform is an isometric
isomorphism $A^p(\R) \to L^p(\R)$,
which allows us to reformulate
\thmref{thm:main}
 as a result about Schauder frames of weighted
exponentials in $A^p(\R)$.

\begin{thm}
\label{main_2}
Let $p > (1+\sqrt{5})/2$. There exist 
 $w \in A^p(\R)$, $\{h^*_j\} \sbt (A^p(\R))^*$, and a 
 uniformly discrete real sequence $\{\lambda_j\}_{j=1}^\infty$
 satisfying $|\lam_j| = n_j + o(1)$, where $n_j$ are
 positive integers,  $0 < n_1 < n_2 < \dots$,
such that every $f\in A^p(\R)$ admits a series expansion
\begin{equation}
\label{eq:ser.w.1}
f(t) =\sum_{j=1}^\infty h^*_j(f)  w(t) e^{2\pi i \lambda_j t} 
\end{equation}
convergent in the $A^p(\R)$  norm.
\end{thm}

Note that $\{n_j\}$ is allowed to be a \emph{subsequence} of the 
positive integers, since we may add more elements with zeros 
as coefficient functionals, and the
 series expansion \eqref{eq:ser.w.1} will remain valid.
 Hence \thmref{thm:main} follows from   \thmref{main_2}.

The restriction $p > (1+\sqrt{5})/2$ is  needed
so that we can invoke 
\lemref{lemP3.3}, while otherwise 
we only use the assumption $p>1$ in the proof.

\subsection{}
We begin the proof by an application of
\thmref{thm:lev24}, which implies the existence of
a function  $u_0 \in \S(\R)$ and
a real sequence  $\{\sig(n)\}_{n=1}^{\infty}$
satisfying 
\begin{equation}
\label{eq:disten}
\sig(n) = n + o(1), \quad n \to + \infty,
\end{equation}
such that for every $N$ the system
$\{u_0(t) e^{2 \pi i \sig(n) t}\}$, $n > N$, 
is complete in $A^p(\R)$.

Next, we choose a normalized Schauder basis 
$\{\varphi_k\}_{k=1}^{\infty}$ 
for the space $A^p(\R)$. (For example, one may take
$\{\ft{\varphi}_k\}_{k=1}^{\infty}$ to be
the normalized basis of Haar functions in $L^p(\R)$).

We will  now construct by induction a sequence of 
Schwartz functions $\{u_k\}$ on $\R$. 
We will perform the construction 
in such a way that for each $k$ and every $N$, the system
\begin{equation}
\label{eq:ukexp}
\{u_k(t) e^{2 \pi i \sig(n) t}\}, \; n > N,
\end{equation}
is complete in the space $A^p(\R)$.
The construction is done as follows.

At the $k$'th step of the induction, given any $ \eta_k >0$ and
any   positive integer $N_k$,
we use the completeness of the system
$\{u_{k-1}(t) e^{2 \pi i \sig(n) t}\}$, $n > N_k$,
to find a trigonometric polynomial
\begin{equation}
\label{eq:defqk}
Q_k(t) = \sum_{N_k < n < N'_k} d_{n,k} e^{2 \pi i \sig(n) t}
\end{equation}
such that
\begin{equation}
\label{eq:phikapprox}
\| \varphi_k  - u_{k-1} \cdot  Q_k \|_{A^p(\R)} < \eta_k.
\end{equation}

We choose a small number $\eps_k > 0$ so that
\begin{equation}
\label{eq:epsksmall}
\eps_k \cdot  \Big(  1  + \nmb{u_{k-1}} \Big) 
\Big( 1 +  \|  \ft{Q}_k \|_{1}
+ \sum_{j=1}^{k-1}  \|  \ft{P}_j  \|_{1} \cdot  \|  \ft{Q}_j \|_{1} \Big)  < 2^{-k} \eta_k,
\end{equation}
and apply \lemref{lemP3.3} in order  to find real trigonometric polynomials
$P_k$ and $\gam_k$ with integer spectrum, such that
\begin{enumerate-num}
\item \label{ktlo:i} $\ft{P}_k(0) = 0$,  $\| \ft{P}_k \|_{\infty} < \eps_k$;
\item \label{ktlo:ii} $\| \gam_k  - 1 \|_{A^p(\T)} < \eps_k$;
\item \label{ktlo:iii} $\|\gam_k \cdot P_k - 1 \|_{A^p(\T)} < \eps_k$;
\item \label{ktlo:iv} $\max\limits_l \|\gam_k \cdot S_l(P_k)  \|_{A^p(\T)} < C_p$.
\end{enumerate-num}

We now choose a large positive integer $\nu_k$ (to be specified later)  and set
\begin{equation}
\tilde{P}_k(t) := P_k(\nu_k t), \quad
\tilde{\gam}_k(t) := \gam_k(\nu_k t).
\end{equation}

Define  $u_k := u_{k-1} \cdot \tilde{\gam}_k$
which is a Schwartz function on $\R$.
We claim that for every $N$ the system \eqref{eq:ukexp}
is complete in the space $A^p(\R)$. Indeed,
let $\alpha$ be a tempered distribution belonging to the
dual space
$(A^p(\R))^* = A^{p'}(\R)$, $p' = p/(p-1)$,
and suppose that $\alpha$ annihilates the system \eqref{eq:ukexp}.
 This means that 
 $\alpha \cdot \tilde{\gam}_k$,
 which also lies in $A^{p'}(\R)$, 
 annihilates the system 
 	$\{u_{k-1}(t) e^{2 \pi i \sig(n) t}\}$, $n >N$.
  By the completeness of the latter system in $A^p(\R)$,
it follows that $\alpha \cdot \tilde{\gam}_k = 0$. 
In turn, using \lemref{lem:uniqaqprod} we conclude that
$\alpha = 0$. Hence the system \eqref{eq:ukexp}
is complete in $A^p(\R)$.

\subsection{}
It follows from \ref{ktlo:ii} and \eqref{eq:epsksmall} that
\begin{equation}
\| u_k - u_{k-1} \|_{A^p(\R)} = \| u_{k-1} (\tilde{\gam_k} - 1) \|_{A^p(\R)}
\le \nmb{u_{k-1}} \| \gam_k - 1  \|_{A^p(\T)} < 2^{-k},
\end{equation}
hence the sequence $u_k$ converges 
in the space $A^p(\R)$ to some element $w \in A^p(\R)$.

\subsection{}
Next we claim that the estimate
\begin{equation}
\label{eq:wpertphk}
\| \varphi_k - w \cdot  \tilde{P}_k \cdot Q_k \|_{A^p(\R)} < 3 \eta_k
\end{equation}
holds for all $k$.   Indeed,
\begin{align}
& \| \varphi_k - w \cdot  \tilde{P}_k \cdot Q_k \|_{A^p(\R)}  \label{eq:estphkp1}
\le \| \varphi_k - u_{k-1} \cdot  Q_k \|_{A^p(\R)} \\[4pt]
& \qquad + \| u_{k-1} \cdot Q_k \cdot (1 -  \tilde{\gam}_k  \cdot \tilde{P}_k ) \|_{A^p(\R)} 
+ \| (u_k - w)  \cdot \tilde{P}_k \cdot Q_k   \|_{A^p(\R)}. \label{eq:estphkp2}
\end{align}
The right hand side of \eqref{eq:estphkp1} is less than $\eta_k$ due to
\eqref{eq:phikapprox}. To estimate the first term in \eqref{eq:estphkp2} 
we   apply inequalities \eqref{est_prod_2} and \eqref{est_prod_3}, and 
use \ref{ktlo:iii} and \eqref{eq:epsksmall}, to obtain
\begin{equation}
\| u_{k-1} \cdot Q_k \cdot (1 -  \tilde{\gam}_k  \cdot \tilde{P}_k ) \|_{A^p(\R)} 
\le \nmb{u_{k-1}} \cdot \|\ft{Q}_k\|_{1}   \cdot \| 1 - \gam_k  \cdot P_k  \|_{A^p(\T)} < \eta_k.
\end{equation}
It remains to estimate the second term in \eqref{eq:estphkp2}. 
We observe that for any fixed $k$ we have
$u_j  \cdot \tilde{P}_k \cdot Q_k \to
w  \cdot \tilde{P}_k \cdot Q_k$ as $j \to + \infty$ in the $A^p(\R)$ norm, hence
\begin{align}
& \label{telesc1}\| (u_k - w)  \cdot \tilde{P}_k \cdot Q_k   \|_{A^p(\R)} 
\le \sum_{j=k+1}^{\infty}
\| (u_{j-1} - u_j)  \cdot \tilde{P}_k \cdot Q_k   \|_{A^p(\R)} \\
& \label{telesc2}\qquad    = \sum_{j=k+1}^{\infty}
\|  u_{j-1}  \cdot (1 - \tilde{\gam}_j)  \cdot \tilde{P}_k \cdot Q_k   \|_{A^p(\R)} \\
& \label{telesc3}\qquad    \le \sum_{j=k+1}^{\infty}
\nmb{u_{j-1}}  \cdot \| \ft{Q}_k \|_1 \cdot \|  \ft{P}_k  \|_{1} 
\cdot \|  1 -  \gam_j   \|_{A^p(\T)}  \\
&\label{telesc4} \qquad    \le \sum_{j=k+1}^{\infty} 2^{-j} \eta_j < \eta_k,
\end{align}
again using \ref{ktlo:ii} and \eqref{eq:epsksmall}, and
assuming (as we may do) that the sequence $\{\eta_k\}$ is decreasing.
The estimate \eqref{eq:wpertphk} thus follows.

\subsection{}
Let  $\{\varphi^*_k\}_{k=1}^{\infty}$ 
be the sequence of continuous linear functionals on $A^p(\R)$
which is biorthogonal to the Schauder basis  $\{\varphi_k\}_{k=1}^{\infty}$.
It follows from \eqref{eq:wpertphk} that if
we  choose the sequence $\{\eta_k\}$ to satisfy
3 $\sum_{k=1}^{\infty} \|  \varphi^*_k \| \cdot \eta_k < 1$, then the system
\begin{equation}
\label{eq:S7}
\{ w \cdot  \tilde{P}_k \cdot Q_k  \}_{k=1}^{\infty} 
\end{equation}
forms another Schauder basis in the space $A^p(\R)$ 
(see \cite[Section 1.9]{You01}).
Hence every $f \in A^p(\R)$ has a series expansion 
\begin{equation}
\label{eq:S1}
f = \sum_{k=1}^{\infty} \psi_k(f) \, w \cdot  \tilde{P}_k \cdot Q_k
\end{equation}
where $\{\psi_k\}$ are the continuous linear functionals 
on $A^p(\R)$ which are biorthogonal to the system
\eqref{eq:S7}.  Moreover, recall that we have chosen the Schauder basis 
$\{\varphi_k\}_{k=1}^{\infty}$ to be normalized,
so it follows from \eqref{eq:wpertphk} that
the norms of the
elements of the system
\eqref{eq:S7} are bounded from below. 
This implies that 
\begin{equation}
\label{eq:S2}
\sup_{k} |\psi_k(f)| \le K \|f\|_{A^p(\R)}, \quad
\lim_{k \to \infty} \psi_k(f)  = 0
\end{equation}
where $K$ is a constant not depending on $f$
(see \cite[Section 1.6]{You01}).

\subsection{}
Notice that we have
\begin{equation}
\label{eq:lacex1}
\tilde{P}_k(t)  Q_k(t) =  P_k(\nu_k t) Q_k(t) = 
\sum_{m \ne 0} \ft{P}_k(m) Q_{k,m}(t),
\end{equation}
where $Q_{k,m}$ are trigonometric polynomials defined by
\begin{equation}
\label{eq:lacex2}
Q_{k,m}(t) := Q_k(t) e^{2 \pi i m \nu_k t} = 
\sum_{N_k < n < N'_k} d_{n,k} e^{2 \pi i (\sig(n) + m \nu_k) t}.
\end{equation} 
If we choose the sequence $\{\nu_k\}$ increasing sufficiently fast, then the spectra
of the polynomials $Q_{k,m}$ follow each other, meaning that
\begin{equation}
\label{eq:lacex3}
\max \spec (Q_{k,m_1}) < \min  \spec (Q_{k,m_2}), \quad m_1 < m_2.
\end{equation}
Moreover, there is a positive, increasing sequence $\{R_k\}$ such that
\begin{equation}
\label{eq:lacex9}
\spec (\tilde{P}_{k} \cdot Q_{k}) = \bigcup_{m \in \spec(P_k)} \spec(Q_{k,m})
\sbt (-R_{k+1}, -R_k) \cup (R_k, R_{k+1}).
\end{equation}

We now define 
\begin{equation}
\label{eq:lamdefunion}
\Lam   := \bigcup_{k =1}^{\infty}  \spec ( \tilde{P}_{k} \cdot Q_{k}),
\end{equation}
then each point $\lam \in \Lam$ has a unique representation as
\begin{equation}
\label{eq:lamptform}
\lam   =  m \cdot \nu_k + \sig(n), \quad k \ge 1,  \quad m \in \spec(P_k), \quad N_k < n < N'_k.
\end{equation}
We can use \eqref{eq:disten} to
choose $\{N_k\}$ increasing fast enough, so that  (say)
\begin{equation}
\label{eq:signear}
|\sig(n) - n| < \tfrac1{10} \cdot k^{-1}, \quad n > N_k.
\end{equation} 
This implies that $\Lam$ is a uniformly discrete set.
Moreover, if the elements of $\Lam$ are
enumerated as $\{\lam_j\}_{j=1}^{\infty}$ 
by increasing modulus, that is,
$0 < |\lam_1| \le |\lam_2| \le \dots$, then
there are positive integers $n_j$ such that
\begin{equation}
\label{eq:lambenum}
|\lam_j| = n_j + o(1), \quad j \to + \infty.
\end{equation}
(We note that each $\lam_j$ may be either positive or negative,
since the polynomials $P_k$ have both positive and negative spectra).

We now observe that in fact we have
$0 < n_1 < n_2 < \dots$. Indeed, 
the fact that the polynomials $Q_k$
have positive spectra implies that if 
$\lam \in \Lam$ is given by  \eqref{eq:lamptform}, then
$|\lam| =  |m| \cdot \nu_k + \sign(m) \cdot \sig(n)$,
where $\sign(m)$ is $ +1$ or $-1$ according to
whether $m$ is positive or negative.
Hence, if  $\lam_j $ and $\lam_l$  are
two distinct points of $\Lam$ then 
$n_j \neq n_l$.

\subsection{}
For $\lam \in \Lam$  given by \eqref{eq:lamptform} 
we define $h_\lam^* \in (A^p(\R))^*$ by
\begin{equation}
\label{eq:coefffunc}
h_\lam^* :=  d_{n,k}  \ft{P}_k(m) \, \psi_k.
\end{equation}
We will prove that each
$f \in A^p(\R)$ admits a series representation 
\begin{equation}
\label{eq:schfexp}
f (t) = \sum_{j=1}^{\infty}
h_{\lam_j}^*(f) w(t)  e^{2 \pi i \lam_j t}
\end{equation}
where the convergence is in the $A^p(\R)$ norm.
This will establish \thmref{main_2}, and as a consequence,
 \thmref{thm:main} will also be proved.

To prove the representation \eqref{eq:schfexp}, we first observe that
any partial sum of the series
can be decomposed as $S' + S'' + S'''$ where
\begin{equation}
\label{eq:partial1}
S'(t) = \sum_{s=1}^{k-1} \psi_s(f)  w (t)  \tilde{P}_s(t)  Q_s(t),
\end{equation}
\begin{equation}
\label{eq:partial2}
S''(t) =  \psi_k(f) w(t) Q_k(t) S_{l}(P_k)(\nu_k t),
\end{equation}
and
\begin{equation}
\label{eq:partial3}
S''' (t)  =   \psi_k(f) w(t) \big[ \ft{P}_k(l+1) S_r(Q_{k,l+1})(t)
+ \ft{P}_k(-(l+1)) S_r(Q_{k,-(l+1)})(t) \big]
\end{equation}
for some $k$, $l$ and $r$. Indeed,  $S'$ consists of the series
terms with $\lam_j$ belonging to entire ``blocks'' of the form
$\spec ( \tilde{P}_{s} \cdot Q_{s})$, $1 \le s \le k-1$.
Similarly, recalling \eqref{eq:lacex1}, \eqref{eq:lacex2},
one can see that  $S''$ consists of the terms with
$\lam_j$ belonging to entire ``sub-blocks'' of the form
$\spec(Q_{k,m})$, $|m| \le l$.
Finally, $S'''$ consists of the remaining terms with
$\lam_j$ belonging to a part of the last two 
sub-blocks $\spec(Q_{k,l+1})$ and $\spec(Q_{k, -(l+1)})$.

We have $\|f - S'\|_{A^p(\R)} = o(1)$ as $k \to \infty$ due to \eqref{eq:S1}.

In order to estimate $\|S''\|_{A^p(\R)}$ we 
denote $\tilde{S}_{k,l}(t):= S_l(P_k)(\nu_k t)$, then
\begin{equation}
	\label{eq:wsklqksplit}
	\|w \cdot \tilde{S}_{k, l} \cdot Q_k \|_{A^p(\R)}\leq \|(w-u_k)\cdot \tilde{S}_{k,l}\cdot Q_k\|_{A^p(\R)}+ \|u_{k-1} \cdot \tilde{\gamma}_k \cdot  \tilde{S}_{k,l} \cdot Q_k \|_{A^p(\R)}.
\end{equation}
The first summand on the right hand side
can be estimated similarly to the inequalities \eqref{telesc1}--\eqref{telesc4},
using the fact that $\|\ft{S}_{k,l}\|_1 \le \| \ft{P}_k\|_1$.
To estimate the second summand, we use \lemref{sep-specR} 
to conclude that if $\nu_k$ is chosen sufficiently large, then
\begin{equation}
	\|u_{k-1} \cdot \tilde{\gamma}_k\cdot\tilde{S}_{k,l} \cdot Q_k\|_{A^p(\R)} 
	<  \|u_{k-1} \cdot Q_k\|_{A^p(\R)} \| \gamma_k\cdot S_{k,l}\|_{A^p(\T)} + 1.
\end{equation}
It follows from \eqref{eq:phikapprox} that
$\|u_{k-1} \cdot Q_k\|_{A^p(\R)} < 2$,
since the sequence   $\{\varphi_k\}$ is normalized in $A^p(\R)$, 
while according to \ref{ktlo:iv} we have 
$\|\gamma_k\cdot S_{k,l}\|_{A^p(\T)} < C_p$.
We conclude using \eqref{eq:wsklqksplit} that 
$\|w \cdot \tilde{S}_{k, l} \cdot Q_k \|_{A^p(\R)}=O(1)$.
In turn, together with \eqref{eq:S2}, \eqref{eq:partial2} this implies the desired estimate
$\|S''\|_{A^p(\R)}=o(1)$ as $k \to \infty$.

Finally, we estimate $\|S'''\|_{A^p(\R)}$ using  \ref{ktlo:i} and  the
inequalities \eqref{est_prod_2}, \eqref{eq:epsksmall} and obtain
\begin{equation}
\|S'''\|_{A^p(\R)}\le 2 |\psi_k(f)| \cdot \|\ft{P}_k\|_\infty \|\ft{Q}_k\|_{1} 
\|w\|_{A^p(\R)}\le 2 \eta_k \cdot |\psi_k (f)|\cdot  \|w\|_{A^p(\R)},
\end{equation}
which shows that $\|S'''\|_{A^p(\R)}=o(1)$ as $k\to\infty$ as well.

We conclude that \eqref{eq:schfexp} indeed
holds, which completes the proof of \thmref{main_2}.

 As a consequence, \thmref{thm:main} is also established.


\section{Schauder frames of translates and nonnegativity}
\label{sec:nonneg}

\subsection{}
In \cite{PS16} the following question was considered: do there exist (unconditional or not) Schauder bases or Schauder frames in the space $L^p(\R)$ consisting of nonnegative functions? It turns out that it is not difficult to construct a Schauder frame formed by nonnegative functions, while unconditional Schauder frames of nonnegative functions do not exist in any $L^p(\R)$ space,
see again \cite{PS16}.

In the paper \cite{JS15} (which was written after \cite{PS16}) some proofs from \cite{PS16} were simplified, and also a Schauder basis consisting of nonnegative functions in $L^1(\R)$  was constructed. A Schauder basis of nonnegative functions in $L^2(\R)$ was constructed in \cite{FPT21}. The existence of such a basis in $L^p(\R)$, $p\neq 1, 2$, remains open.

Motivated by the recent interest in nonnegative coordinate systems in $L^p(\R)$ spaces, we consider the following question: does there exist a Schauder frame  in the space $L^p(\R)$  formed by a uniformly discrete sequence of translates of a \emph{nonnegative} function? 

In this section, our goal is to show that 
this additional requirement of nonnegativity can indeed be achieved 
in our main result, namely:

\begin{thm}
\label{thm:nonneg}
The function $g$ in \thmref{thm:main}
can be chosen nonnegative.
\end{thm}

To establish this,  we will make certain modifications to
the proof of \thmref{thm:main}.

\subsection{}
\label{subsec:nonnegL24}
First, we recall that our construction 
 in \secref{secC1}  began with an application of
\thmref{thm:lev24}, which yields
a real sequence $\{\lam_n\}_{n=1}^{\infty}$ satisfying
$ \lam_n = n +  o(1)$,  and 
a Schwartz function $g$ on $\R$, 
such that for any $N$ the system 
\begin{equation}
\{g(x - \lam_n)\}, \; n>N,
 \end{equation}
is complete in the space $L^p(\R)$, $p>1$.
So first, we need the following result:

\begin{lemma}
\label{lem:nnglev24}
The function $g$ in
\thmref{thm:lev24}
can be chosen nonnegative.
\end{lemma}

The proof, based on adapting the approach in \cite{Lev25}, is given in 
\cite[Section 4]{LT25b}.

\subsection{}
The second ingredient which enables us to construct
a nonnegative function $g$ in \thmref{thm:main},
 is the following lemma.

\begin{lem}
\label{lemF2.1}
Let $a$ and $h$ be two positive real numbers,
$2h < a < \frac1{2} - 2h$. Then there is a nonnegative
function $\varphi \in A(\T)$ with the following properties:
\begin{enumerate-num}
\item \label{itp:i} $\ft{\varphi}(0)=1$, $\ft{\varphi}(n) \ge 0$ for all $n \in \Z$;
\item \label{itp:ii} $\varphi$ vanishes on the set
$[-a-h, -a + h] \cup [a-h, a+h]$;
\item \label{itp:iii} $\| \varphi - 1\|_{A^p(\T)} \le 12 \cdot h^{(p-1)/p}$
for every $p \ge 1$.
\end{enumerate-num}
\end{lem}

\begin{proof}
We consider the function
\begin{equation}
\label{lemF2.1.1}
\psi(t) := 6 \cdot \Del_h(t) -  \tau_h(t+a) - \tau_h(t-a),
\end{equation}
which is in $A(\T)$. We observe that 
$\tau_h = \Del_h \ast (\del_{-h} + \del_0 + \del_h)$,
where $\del_t$ denotes the Dirac measure at a point $t \in \T$. Hence
\begin{equation}
\label{lemF2.1.2}
\psi = 6  \cdot \Del_h - \Del_h \ast (\del_{-h} + \del_0 + \del_h) \ast (\del_{-a} + \del_a)
= \Del_h \ast (6 \cdot \del_0 - \nu),
\end{equation}
where $\nu = (\del_{-h} + \del_0 + \del_h) \ast (\del_{-a} + \del_a)$.
Then $\nu$ is a positive measure on $\T$ with total mass
 $\ft{\nu}(0) = 6$, and the Fourier coefficients $\ft{\nu}(n)$ are all real valued.
This implies that
\begin{equation}
\label{lemF2.1.3}
-6 \le \ft{\nu}(n) \le 6, \quad n \in \Z.
\end{equation}
It now follows from \eqref{lemF2.1.2} that
$\ft{\psi}(n) = \ft{\Del}_h(n)  (6 - \ft{\nu}(n))$,
and 
since both terms in the product are nonnegative, we conclude
that $\ft{\psi}(n) \ge 0$ for all $n \in \Z$. Furthermore,
\begin{equation}
\label{lemF2.1.4}
\| \psi \|^p_{A^p(\T)}
= \sum_{n \in \Z} \ft{\psi}(n)^p 
= \sum_{n \in \Z} \ft{\Del}_h(n)^p  (6 - \ft{\nu}(n))^p
\le 12^p \| \Del_h \|^p_{A^p(\T)} \le 12^p h^{p-1}
\end{equation}
for every $p \ge 1$, due to \eqref{estim_delta}.
 Finally, set $\varphi := 1 + \psi$. 
Then it is obvious that  condition \ref{itp:i} holds. 
It follows from \eqref{lemF2.1.1} that the function
$\varphi$  is  nonnegative and satisfies \ref{itp:ii}.
Moreover, we obtain   \ref{itp:iii} as a consequence
of \eqref{lemF2.1.4}. The lemma is thus proved.
\end{proof}

\begin{remark}
One can show that condition \ref{itp:ii} implies that 
 $\| \varphi - 1\|_{A^p(\T)} \ge h^{(p-1)/p}$
for every $p \ge 1$, so the estimate 
\ref{itp:iii} is sharp up to the numerical
value of the constant.
\end{remark}

\subsection{}
We now use \lemref{lemF2.1} to strengthen \lemref{lemP3.3} as follows.

\begin{lem}
\label{lemN5.3}
Let $p > (1+\sqrt{5})/2$. 
Given any $\eps > 0$ there exist two real trigonometric polynomials
$P$  and  $\gam$ with integer spectrum, such that 
\begin{enumerate-num}
\item \label{nntlo:pd} $\ft{\gam}(0)=1$, $\ft{\gam}(n) \ge 0$ for all $n \in \Z$;
\item \label{nntlo:i} $\ft{P}(0) = 0$,  $\| \ft{P} \|_{\infty} < \eps$;
\item \label{nntlo:ii} $\| \gam  - 1 \|_{A^p(\T)} < \eps$;
\item \label{nntlo:iii} $\|\gam \cdot P - 1 \|_{A^p(\T)} < \eps$;
\item \label{nntlo:iv} $\max\limits_l \|\gam \cdot S_l(P)  \|_{A^p(\T)} \le 3$.
\end{enumerate-num}
\end{lem}

The novel part of this result compared to \lemref{lemP3.3}
  is the addition of property
\ref{nntlo:pd}, that is, the requirement that $\gamma$ 
have nonnegative Fourier coefficients.

\begin{proof}[Proof of \lemref{lemN5.3}]
We first choose a small $\del = \del(\eps, p) > 0$, and 
then choose and fix $0 < h < 1/8$ and a positive integer $N$ 
(both depending on $\eps$, $\del$ and $p$) such that 
\begin{equation}
\label{eq:nnmainconlem}
N > \eps^{-1}, \quad (1 + 12^p h^{p-1})^N < 1 + \del, \quad
13^p N^{-p} h^{-1}  < 1 - \del.
\end{equation}
To see that such a choice of $h$ and $N$ exists, it suffices
to let $h = h(N,\del,p)$
be defined by the condition $N = 12^{-p}  h^{1-p} \log(1+\del)$,
and take $N$ to be sufficiently large.

We now choose  an arbitrary number $a$  satisfying
$2h < a < \frac1{2} - 2h$.
Let $\varphi$ be the function of \lemref{lemF2.1},
and let $\psi = 1 - h^{-1} \Del_h \ast \frac1{2} (\del_{-a} + \del_a)$,
where again $\del_t$ is the Dirac measure at a point $t \in \T$. 
We observe that 
$\varphi \cdot \psi = \varphi$. 
Let  $f$ and $g$ be  Fourier partial sums 
of $\varphi$ and $\psi$ respectively,  
of sufficiently high order such that
$\|f \cdot g - f \|_{A} < \del$.

Next,  choose a  positive integer $\nu$ satisfying 
$\nu > 2( \deg(f)+ \deg(g))$,
and define
\begin{equation}
\gam(t) := \prod_{j=0}^{N-1} f(\nu^j t), \quad
P(t) := \frac1{N} \sum_{j=0}^{N-1} g(\nu^j t).
\end{equation}		
The proof can now continue in the same way as in
\lemref{lemP3.3}, and it follows that this
choice of $\gam$ and $P$ indeed satisfies
all the conditions \ref{nntlo:pd}--\ref{nntlo:iv}
of \lemref{lemN5.3}.
\end{proof}

\subsection{}
We can now adjust the proof of \thmref{thm:main}
given in \secref{secC1} 
as follows. Recall that in the proof we have
constructed by induction a sequence of 
Schwartz functions $u_k$. 
Thanks to \lemref{lem:nnglev24}
we may now assume that the first Schwartz function $u_0$
has a nonnegative Fourier transform $\ft{u}_0$.
Next, at the $k$'th step of the induction
we now use \lemref{lemN5.3} instead of
\lemref{lemP3.3}, which yields  a trigonometric polynomial
$\gam_k$ with nonnegative Fourier coefficients.
As a consequence, it follows that the function
$u_k := u_{k-1} \cdot \tilde{\gam}_k$
is a Schwartz function whose
Fourier transform $\ft{u}_k$ is nonnegative.

In turn,  the sequence $u_k$ 
converges in the space $A^p(\R)$
 to an element $w \in A^p(\R)$ whose Fourier
transform $\ft{w}$ is a nonnegative function
in $L^p(\R)$. We thus conclude that
\thmref{main_2} holds with the extra condition 
that $w$ have a nonnegative Fourier transform.
Finally, this means that \thmref{thm:main}
holds  with the function $g = \ft{w}$ which is
nonnegative,
 and so \thmref{thm:nonneg} is established.
\qed

\subsection*{Remarks}
1. In the space  $L^p(\R)$, $p>2$, one
 can adapt the technique 
from \cite[Section 3]{FOSZ14} in order to 
 construct a Schauder frame (not unconditional)
  formed by an \emph{arbitrary unbounded} 
sequence of translates of a \emph{nonnegative} function $g$.

2. One may also consider the
   Banach space  $C_0(\R)$ of continuous functions on $\R$ vanishing 
at infinity, endowed  with the norm  $\|f\|_{\infty} = \sup |f(x)|$, $x \in \R$.
First we  note that this space \emph{does not admit any 
unconditional Schauder frames}.  Indeed, a Banach space with an
unconditional Schauder frame is isomorphic to a complemented subspace of
a Banach space with an unconditional Schauder basis \cite{CHL99},
which is not the case for the space $C_0(\R)$
(see \cite[Section II.D, Corollary 12]{Woj91}).
On the other hand, again based on the technique 
from \cite[Section 3]{FOSZ14}, one can show
that the space $C_0(\R)$ admits a Schauder frame 
(not unconditional)   formed by an arbitrary unbounded 
sequence of translates of a function $g$, which can
moreover be chosen nonnegative.


\end{document}